\documentclass[a4paper,11pt]{amsart}
\title[transformation formulas]{On transformation formulas of $p$-adic hypergeometric functions }

\usepackage{amssymb}
\usepackage{amsthm}
\usepackage{mathrsfs}
\usepackage{amsmath}  
\usepackage{graphics}  
\usepackage{centernot}
\usepackage{mathtools}
\usepackage{tikz-cd}
\usepackage{xcolor}
\usepackage{hyperref}

\author{Wang Chung-Hsuan}
\address{Independent Researcher, Tainan, Taiwan}
\email{a78sddrt@gmail.com}
\date{}

\newtheorem{defi}{Definition}[section]
\newtheorem{thm}[defi]{Theorem}
\newtheorem{lem}[defi]{Lemma}
\newtheorem{cor}[defi]{Corollary}
\newtheorem{conj}[defi]{Conjecture}
\newtheorem{prop}[defi]{Proposition}
\newtheorem{remark}[defi]{Remark}
\newtheorem{eg2}{Example}
\numberwithin{equation}{section}
\tikzset{
    labl/.style={anchor=south, rotate=90, inner sep=.5mm}
}
\begin{document}

\date{}
\maketitle 
\begin{abstract}
 In this paper, we review hypergeometric functions $\mathscr{F}^{\rm Dw}_{\underline{a}}(t),$ $\mathscr{F}^{(\sigma)}_{\underline{a}}(t)$ and $\widehat{\mathscr{F}}^{(\sigma)}
 _{\mathbf{a}}(t)$ together with their conjectured transformation formulas, and show that one transformation formula implies the other.
\end{abstract}

%%%%%%%%%%%%%%%%%%%%%%%%%%%%%%%%%%%%%%%%%%%%%%%%
%%%            Introduction
%%%%%%%%%%%%%%%%%%%%%%%%%%%%%%%%%%%%%%%%%%%%%%%%%%
\section{Introduction}
Let $s\geq 1$ be an integer and $a_1,\cdots,a_s, b_1, \cdots b_{s-1}\in \Bbb{Z}_p$ of $p$-adic integers.
The classical hypergeometric function of one-variable
is defined to be the power series
$$
{}_sF_{s-1}\left({a_1,\ldots,a_s\atop b_1,\ldots,b_{s-1}};t\right)=\sum\limits_{k=0}^{\infty}\frac{(a_1)_k\cdots(a_s)_k}
{(b_1)_k\cdots(b_{s-1})_k}\frac{t^k}{k!}\in \Bbb{Q}_p[[t]]
$$
where $(\alpha)_k$ denotes the Pochhammer symbol $( (\alpha)_0=1, (\alpha)_k=\alpha (\alpha+1)\cdots (\alpha+k-1)$ when $k\geq 1$, cf. \cite{slater}).\medskip

Let $\underline{a}=(a_1,...,a_s)\in \Bbb{Z}_p^s$ be a $s$-tuple of $p$-adic integers, and consider
the power series
$$
F_{\underline{a} }(t)={}_sF_{s-1}\left({a_1,\ldots,a_s\atop 1,\ldots,1};t\right)=\sum\limits_{k=0}^{\infty}\frac{(a_1)_k}{k!}\cdots \frac{(a_s)_k}{k!}t^k.
$$
One can see that this is a power series with
$\mathbb{Z}_p$-coefficients.\medskip

Let $a^\prime$ be the Dwork prime of $a$, which is defined to be $(a+l) / p$ where $l\in \{0,1,...,p-1\}$ is the unique integer such that $a+l \equiv 0\mod p.$ And we write $\underline{a}^\prime=(a_1^\prime, \cdots, a_s^\prime).$\medskip

In the paper \cite{Dw}, B. Dwork introduced his $p$-adic hypergeometric functions $\mathscr{F}^{\rm Dw}_{\underline{a}}(t)$ as 
\[
 \mathscr{F}^{{\rm Dw}}_{\underline{a}}(t):=F_{\underline{a}}(t)/F_{\underline{a}^\prime}(t^p)
\]
(Definition \ref{def_Dwork_HG}) and showed that $\mathscr{F}^{\rm Dw}_{\underline{a}}(t)$ satisfy congruence relation (\cite[p.41, Theorem 3]{Dw}) 
$$ 
\mathscr{F}^{\rm Dw}_{\underline{a}}(t)\equiv \frac{[F_{\underline{a}}(t)]_{<p^n}}{[F_{\underline{a}^\prime}(t^p)]_{<p^n}} \mod{p^n\mathbb{Z}_p[[t]]}
$$ where the notation $[f(t)]_{<m}:= \sum_{n<m} c_n t^n$ denotes
the truncated 
polynomial for a power series $f(t)=\sum{c_n} t^n$. Dwork also showed that his function has some geometric application as follows. Let $p$ be an odd prime and $\alpha\in \mathbb{F}_p\backslash \{0,1\}.$ Suppose the elliptic curve
$$
E_\alpha : y^2=x(1-x)(1-\alpha x)
$$ is ordinary and $\epsilon_\alpha$ is the unit root of $E_\alpha,$ i.e., the root of $x^2-a_px+p$ which is a unit. Then
$$
\epsilon_\alpha=(-1)^{\frac{p-1}{2}}\mathscr{F}_{\frac{1}{2},\frac{1}{2}}^{\rm Dw}(\widehat{\alpha}) 
$$ where $\widehat{\alpha}$ is the Teichm\"{u}ller lift of $\alpha$ (see, for example, \cite[Theorem 0.3.8]{K} or \cite[\S 7]{V}).\medskip

In \cite{A} and \cite{W}, the $p$-adic hypergeoemtric functions of logarithmic type $\mathscr{F}^{(\sigma)}_{\underline{a}}(t)$ (Definition \ref{def_HG_1})
and $\widehat{\mathscr{F}}_{\mathbf{a}}^{\;(\sigma)}(t)$ (Definition \ref{def_HG_2}) where $\mathbf{
a}=(a,\cdots,a)$ are introduced. It is shown that these $p$-adic hypergeometric functions satisfy congruence relations similar to Dwork's (Theorem \ref{cong_rel_1} and Theorem \ref{cong_rel_2}).\medskip

The geometric significance of these functions lies in their connection to syntomic regulators; specifically, their special values explicitly describe the image of certain elements of algebraic $K$-groups under the syntomic regulator maps. Furthermore, this framework serves a bridge to investigate the $p$-adic Beilinson conjecture, linking these hypergeometric values to the special values of $p$-adic $L$-functions of elliptic curves (for details on the construction, geometric applications, and specific evaluations, cf. \cite{A, A2}).\medskip

 Two expected relations between these hypergeometric functions are proposed in \cite{W}. For simplicity, we shall denote the constant $s$-tuple $(a,\cdots,a)$ simply as $a$ when the context is clear. We conjecture the transformation formula of Dwork's $p$-adic hypergeometric functions (Conjecture \ref{Dwork-trans-conj})
 $$
\mathscr{F}^{{\rm Dw}}_{a}(t)=\pm((-1)^st)^{l}\mathscr{F}^{{\rm Dw}}_{a}(t^{-1})
$$
 and the transformation formula of $\mathscr{F}_{a}^{(\sigma)}(t)$ and $\widehat{\mathscr{F}}_{a}^{\;(\sigma)}(t)$ (Conjecture \ref{transformation-conj})
 $$
\mathscr{F}_{a}^{\;(\sigma)}(t)=-\widehat{\mathscr{F}}_{a}^{\;(\widehat{\sigma})}(t^{-1}).
$$ 
 In \cite{W}, some special cases of the transformation formula have been proved (the case $s=2$, $a\in \frac{1}{N}{\mathbb{Z}}$, $0<a<1$ and $p>N$, cf. \cite[Theorem 4.14]{W} and \cite[Theorem 4.20]{W}). This is proved geometrically by using hypergeometric curves (\cite[\S 4]{W}).\medskip

 In 2025, Nemoto generalized our result of Dwork's case (\cite{N}). He proved the following theorem.
Suppose that the following two conditions:
\begin{enumerate}
    \item  $a\in N^{-1}\Bbb{Z}$ and $p\nmid N$ for some $N\geq 2$
    \item $0<a<1$
\end{enumerate}

Then Conjecture \ref{Dwork-trans-conj} is true when $s\geq 2$. Furthermore, if $p=2$, $$
\mathscr{F}^{{\rm Dw}}_{a}(t)=-((-1)^st)^{l}\mathscr{F}^{{\rm Dw}}_{a}(t^{-1})
$$ holds if and only if $s$ is odd and $a^\prime \equiv 1 \pmod{2}$.
 
In this paper, we show that these two transformation formulas are closely related. Specifically, our main result (Theorem \ref{conj2_to_conj1}) establishes that transformation formula of Dwork's $p$-adic hypergeometric functions implies transformation formula of $\mathscr{F}_{a}^{(\sigma)}(t)$ and $\widehat{\mathscr{F}}_{a}^{\;(\sigma)}(t)$. 
As an immediate corollary, we deduce that Conjecture \ref{transformation-conj} holds when $s=1$, or when $s\geq 2$ with $0<a<1$, $a\in N^{-1}\mathbb{Z}$, and $p\nmid N$ for some $N\geq 2$. \medskip

This paper is organized as follows. In \S 2, we recall the definition of $p$-adic digamma functions and $p$-adic Euler constant which we need to define the hypergeometric functions. In \S 3, we recall the definition and some properties of these $p$-adic hypergeometric functions. In \S 4, we prove some relation between coefficients of $p$-adic hypergeometric functions of logarithmic type. One of our main result is Lemma \ref{beta+hbeta}, which provides a relation between the coefficients 
of $\mathscr{F}_{a}^{(\sigma)}(t)$ and $\widehat{\mathscr{F}}_{a}^{\;(\sigma)}(t)$. In \S 5, we first recall the conjectures between hypergeometric functions  $\mathscr{F}_{a}^{\;(\sigma)}(t)$ and $\widehat{\mathscr{F}}_{a}^{\;(\sigma)}(t)$ and then move on to the transformation formula of Dwork's $p$-adic hypergeometric functions. Finally, we give the proof of our main result (Theorem \ref{conj2_to_conj1}) using the relation in \S 4 (Lemma \ref{beta+hbeta}).

\medskip

\textbf{Acknowledgement.} I would like to thank for the help of Professor Po-Yi Huang. He gave a lot of helpful advice and comments on this paper.  

%%%%%%%%%%%%%%%%%%%%%%%%%%%%%%%%%%%%%%%%%%%%%%%%%%%%%
%%%%%%%%%%%%%              Definition             %%%
%%%%%%%%%%%%%%%%%%%%%%%%%%%%%%%%%%%%%%%%%%%%%%%%%%%%%

\section{$p$-adic Digamma functions and $p$-adic Euler Constant}\label{digamma}
Let us recall the definition of the $p$-adic digamma functions in \cite[\S 2.2]{A} which appear in the definition of $p$-adic hypergeometric functions.\medskip

Let $r\in \mathbb{Z}$ and $z\in \mathbb{Z}_p,$ we define
$$
\widetilde{\psi}_p^{(r)}(z)=\lim\limits_{n\in \mathbb{Z}_{\geq >0}, n\rightarrow z} \sum\limits_{1\leq k<n, p\nmid k}\frac{1}{k^{r+1}}.
$$

These functions are $p$-adic continuous function on $\mathbb{Z}_p$ since
$$
\sum\limits_{1\leq k<p^s, p\nmid k}{k^m}\equiv\left\{
%\begin{array}{rcl} 
%\begin{eqnarray}
\begin{aligned}
&-p^{s-1} \quad &p\geq 3 \; \text{and} \; (p-1) | m \\
&2^{s-1}  \quad &p=2  \; \text{and}  \; 2|m \\
&1  \quad &p=2  \; \text{and} \; s=1 \\
&0  \quad &\text{otherwise}
\end{aligned}
%\end{eqnarray}
%\end{array}
\right.
$$
modulo $p^s.$\medskip
\begin{defi}\label{padic_euler_const}
We define \textbf{$p$-adic Euler constant} 
$$
\gamma_p=-\lim\limits_{s\rightarrow \infty}\frac{1}{p^s}\sum\limits_{0\leq j<p^s, p\nmid j}\log(j)
$$ where $\log$ is Iwasawa logarithm.\medskip
\end{defi}

\begin{defi}\label{padic_digamma}
We define the \textbf{$p$-adic digamma function} to be
$$
\psi_p(z)=-\gamma_p+\widetilde{\psi}_p^{(r)}(z).
$$\medskip
\end{defi}

Here we list some formulas on $p$-adic digamma function.
\begin{thm}\label{formula_digamma}
$(1)$ $\psi_p(0)=\psi_p(1)=-\gamma_p.$ \medskip

$(2)$ $\psi_p(z)=\psi_p(1-z).$ \medskip

$(3)$ 
$\psi_p(z+1)-\psi_p(z)=z^{-1}$ if $z\in \mathbb{Z}_p^\times$ and $0$ if $z\in p\mathbb{Z}_p.$
\end{thm}
\begin{proof}
See \cite[Theorem 2.4]{A}.
\end{proof}

\section{$\mathscr{F}_{\underline{a}}^{\rm Dw}(t)$,  $\mathscr{F}_{\underline{a}}^{\;(\sigma)}(t)$ and $\widehat{\mathscr{F}}_{\mathbf{a}}^{\;(\sigma)}(t)$}\label{definition}
In this section, we review the definition of $p$-adic hypergeometric functions $\mathscr{F}_{\underline{a}}^{\rm Dw}(t)$, $\mathscr{F}_{\underline{a}}^{\;(\sigma)}(t)$ and $\widehat{\mathscr{F}}_{\mathbf{a}}^{\;(\sigma)} (t)$ and some of their properties.\medskip

We start with the definition of $p$-adic hypergeometric power series. 
\begin{defi}
 Let $s\geq 1$ be an integer and $\underline{a}=(a_1,...,a_s)$ be a $s$-tuples  of $p$-adic integers. We define
the \textbf{$p$-adic hypergeometric power series}
$$
F_{\underline{a}}(t):=\sum\limits_{k=0}^{\infty}\frac{(a_1)_k}{k!}\cdots \frac{(a_s)_k}{k!}t^k
$$ where $(\alpha)_k$ denotes the Pochhammer symbol $( (\alpha)_0=1, (\alpha)_k=\alpha (\alpha+1)\cdots (\alpha+k-1)$ when $k\geq 1)$. Since
$$
\frac{(\alpha)_k}{k!}=(-1)^{k}\binom{-\alpha}{k},
$$ the $p$-adic hypergeometric power series is a power series with $\mathbb{Z}_p$-coefficients.\medskip

In particular, if all components of the tuple are identical, i.e., $\underline{a}=(a,\cdots,a)\in \Bbb{Z}_p^s$, we write $\mathbf{a}=(a,\cdots,a)$ and $F_{\mathbf{a}}(t)=F_{\underline{a}}(t)$ to denote this symmetric case.
\end{defi}

\begin{eg2}
Let $s=1$ and write $a=a_1.$ Then 
$$
F_a(t)=(1-t)^{-a}.
$$
\end{eg2}

\begin{defi}
 For any $a \in \mathbb{Z}_p$, we define the \textbf{Dwork prime} of $a$ to be $a^\prime=(a+l) / p$ where $l\in \{0,1,...,p-1\}$ is the unique integer such that $a+l \equiv 0\mod p $. \medskip
 
 Inductively, we define the $i$-th Dwork prime $a^{(i)}$ to be $(a^{(i-1)})^\prime$ with $a^{(0)}=a.$
\end{defi}

Now we can define Dwork's $p$-adic hypergeometric functions as follow.
\begin{defi}[{\cite{Dw},\textbf{Dwork's $p$-adic hypergeometric functions}}]\label{def_Dwork_HG}
Let $s \geq 1$ be an integer. For $(a_1,\cdots,a_s)\in \mathbb{Z}_p^s$, Dwork's $p$-adic hypergeometric function is defined as
$$
 \mathscr{F}^{{\rm Dw}}_{\underline{a}}(t):=F_{\underline{a}}(t)/F_{\underline{a}^\prime}(t^p)
$$
where $F_{\underline{a}}(t)$ and $F_{\underline{a}^\prime}(t)$ are hypergeometric power series.
\end{defi}

We give a simple example of Dwork's $p$-adic hypergeometric functions. 
\begin{eg2}
Let $\underline{a}=(1,1,\cdots,1).$ Since the Dwork prime of $1$ is $1,$ we have
$$
\mathscr{F}^{\rm Dw}_{\underline{a}}(t)=F_{\underline{a}}(t)/F_{\underline{a}^\prime}(t^p)=\frac{\sum t^n}{\sum t^{pn}}=\frac{\frac{1}{1-t}}{\frac{1}{1-t^p}}=1+t+\cdots+t^{p-1}.
$$
\end{eg2}

We now introduce the definitions of the other two $p$-adic hypergeometric functions, which are known as $p$-adic hypergeometric functions of logarithmic type, proposed by Asakura (\cite[Definition 3.1]{A}) and Wang (\cite[Definition 2.1]{W}). To state their definitions, we first set up the general algebraic framework.\medskip

Let $W=W(\overline{\mathbb{F}}_{p})$ denote the Witt ring, and $K$=Frac$W$ its fractional field. Let $\sigma : W[[t]] \rightarrow W[[t]]$ be a $p$-th Frobenius given by $\sigma(t)=ct^p$ with $c\in 1+pW$:
\begin{equation}
\biggl(\sum_{i}a_{i}t^i\biggr)^{\sigma}=\sum_{i}a_{i}^{F}c^it^{ip}
\end{equation}
where $F: W \rightarrow W$ is the Frobenius on $W$.\medskip

Put
$$ q:=\left\{
\begin{aligned}
4 &\quad p=2 \\
p  &\quad p\geq3. \\
\end{aligned}
\right.
$$
and
$$
e:=l^\prime-\lfloor \frac{l^\prime}{p} \rfloor.
$$
where $l^\prime\in \{0,1,...,q-1\}$ be the unique integer such that $a+l^\prime \equiv 0\mod q $. \medskip
%%%%%%%%%%%%%%%%%%%%%%%%%%%%%%%%%%%%%%%%%%%%
%% def of p-adic HGF of log type
%%%%%%%%%%%%%%%%%%%%%%%%%%%%%%%%%%%%%%%%%%%%%
We now recall the definition of $p$-adic hypergeometric functions.

\begin{defi}[{\cite[Definition 3.1]{A}}, \textbf{$p$-adic hypergeometric functions of logarithmic type}]\label{def_HG_1}
Let $s\geq 1$ be a positive integer. Let $\underline{a}=(a_1,\cdots, a_s)\in \mathbb{Z}_p$ and $\underline{a}^\prime=(a_1^\prime,\cdots,a_s^\prime)$ where $a_{i}^\prime$ denotes the Dwork prime. Let $\sigma : W[[t]]\rightarrow W[[t]]$ be the $p$-th Frobenius endomorphism given by $\sigma(t)=ct^p$ with $c\in 1+pW.$ Then we define the $p$-adic hypergeometric functions of logarithmic type to be
$$
\mathscr{F}^{(\sigma)}_{\underline{a}}(t):=\frac{G^{(\sigma)}_{\underline{a}}(t)}{F_{\underline{a}}(t)}
$$ with
$$
G^{(\sigma)}_{\underline{a}}(t)=\psi_p(a_1)+\cdots+\psi_p(a_s)+s\gamma_p-p^{-1}\log(c)+\int_{0}^{t}(F_{\underline{a}}(t)-F_{\underline{a}^\prime}(t^{\sigma}))\frac{dt}{t}
$$ where
$\psi_p(z)$ is the $p$-adic digamma function $($Definition \ref{padic_digamma}$)$, $\gamma_p$ is the $p$-adic Euler constant $($Definition \ref{padic_euler_const}$)$ 
and $\log(z)$ is the Iwasawa logarithmic function.
\end{defi}
%%%%%%%%%%%%%%%%%%%%%%%%%%%%%%%%%%%%%%%%%%%%%%%%
%%  def of p-adic HGF of t^-1
%%%%%%%%%%%%%%%%%%%%%%%%%%%%%%%%%%%%%%%%%%%%%%%%
\begin{defi}[{\cite[Definition 2.1]{W}}]\label{def_HG_2}
Let $s\geq 1$ be a positive integer, $a\in\mathbb{Z}_p \backslash \mathbb{Z}_{\leq 0}$ and $\sigma : W[[t]] \to W[[t]]$ be the $p$-th Frobenius endomorphism given by $\sigma(t)=ct^p$ with $c\in 1+pW$. Fix $\underline{a}=(a,\cdots,a).$ Then we define
$$
\widehat{\mathscr{F}}_{\mathbf{a}}^{\;(\sigma)}(t):=\frac{\widehat{G}^{(\sigma)}_{\mathbf{a}}(t)}{F_{\mathbf{a}}(t)}
$$ where
$$
\widehat{G}^{(\sigma)}_{\mathbf{a}}(t)=t^{-a} \int_{0}^{t} (t^aF_{\mathbf{a}}(t)-(-1)^{se}[t^{a^\prime}F_{\mathbf{a}^\prime}(t)]^{\sigma})\frac{dt}{t}
$$ and $\mathbf{a}=(a,\cdots,a), \ \mathbf{a}^\prime=(a^\prime,\cdots, a^\prime)$ are $s$-tuples.
\end{defi}

\begin{remark}
Here we think $t^{\alpha}$ to be an abstract symbol with relations $t^{\alpha}\cdot t^{\beta}= t^{\alpha+\beta}$, on which $\sigma$ acts by $\sigma(t^{\alpha})=c^{\alpha}t^{p\alpha}$, where 
$$
c^{\alpha}=(1+pu)^{\alpha}:=\sum_{i=0}^{\infty}\binom{\alpha}{i}p^{i}u^i, \quad u\in W.  
$$
Moreover we think $\int_{0}^{t}(-)\frac{dt}{t}$ to be a operator such that
$$
\int_{0}^{t}t^{\alpha}\frac{dt}{t}=\frac{t^{\alpha}}{\alpha}, \quad \alpha \neq 0. 
$$
\end{remark}

%%%%%%%%%%%%%%%%%%%%%%%%%%%%%%%%%%%%%%%%%%%%%%%%%%%%%%%%%%%%%%%%%%%%%%%%%%%
%% Congruecne relation for p-adic HGFs
%%%%%%%%%%%%%%%%%%%%%%%%%%%%%%%%%%%%%%%%%%%%%%%%%%%%%%%%%%%%%%%%%%%%%%%%%%%

Now let us recall the so-called congruence relations for these hypergeometric functions. Using this property, one can define special values of the function.\medskip

For a power series $f(t)=\sum^{\infty}_{i=0}a_it^i\in W[[t]]$, we denote $f(t)_{<n}:=\sum_{i=0}^{n-1}a_it^i$ the truncated polynomial. Then the congruence relations for these hypergeometric functions are given as follows.

\begin{thm}[\textbf{congruence relation of $\mathscr{F}^{\rm Dw}_{\underline{a}}(t)$}]\label{cong_rel_Dw}
Let $\underline{a}=(a_1,\cdots, a_s)$ with $a_i\in \mathbb{Z}_p$. We have
$$
\mathscr{F}^{\rm Dw}_{\underline{a}}(t)\equiv \frac{F_{\underline{a}}(t)_{<p^n}}{[F_{\underline{a}^\prime}(t^p)]_{<p^n}} \mod p^n\mathbb{Z}_p[[t]]
$$
\end{thm} 
\begin{proof}
See \cite[p.37, Theorem 2]{Dw}.
\end{proof}

\begin{thm}[\textbf{congruence relation of $\mathscr{F}^{(\sigma)}_{\underline{a}}(t)$}]\label{cong_rel_1}
Let $\underline{a}=(a_1,\cdots, a_s)$ with $a_i \in \mathbb{Z}_p\backslash \mathbb{Z}_{\leq 0}.$ Then
$$
\mathscr{F}^{(\sigma)}_{\underline{a}}(t)\equiv \frac{G^{(\sigma)}_{\underline{a}}(t)_{<p^n}}{F_{\underline{a}}(t)_{<p^n}} \mod p^nW[[t]]
$$ if $c\in 1+qW.$ If $p=2$ and $c\in 1+2W$, then the congruence holds modulo $p^{n-1}.$
\end{thm}
\begin{proof}
This is \cite[Theorem 3.3]{A}.
\end{proof}

\begin{thm}[\textbf{congruence relation of $\widehat{\mathscr{F}}^{\; (\sigma)}_{\mathbf{a}}(t)$}]\label{cong_rel_2}
Let $\mathbf{a}=(a,\cdots,a)$ with $a\in \mathbb{Z}_p\backslash \mathbb{Z}_{\leq 0}$ and suppose that $c\in 1+qW.$ Then
$$
\widehat{\mathscr{F}}^{\; (\sigma)}_{\mathbf{a}}(t)\equiv \frac{\widehat{G}^{(\sigma)}_{\mathbf{a}}(t)_{<p^n}}{F_{\mathbf{a}}(t)_{<p^n}} \mod p^nW[[t]]
$$ for all $n\in \mathbb{Z}_{\geq 0}.$
\end{thm}
\begin{proof}
This is \cite[Theorem 3.1]{W}.
\end{proof}

\begin{cor}\label{overconvergent}
Suppose there exists an integer $r\geq 0$ such that $a_i^{(r)}=a_i$ and  $a^{(r)}=a$ for some $r>0,$ where $(-)^{(r)}$ denote the $r$-th Dwork prime, then
$$
{\mathscr{F}}_{\underline{a}}^{\;(\sigma)}(t)\in W\langle t, t^{-1}, h_1(t)^{-1} \rangle, \quad h_1(t):=\prod\limits_{i=0}^{r-1}F_{\underline{a}^{(i)}}(t)_{<p}
$$
is convergent function, where $W\langle t, t^{-1}, h_1(t)^{-1} \rangle :=\underset{n}{\varprojlim}(W/p^n[t,t^{-1},h_1(t)^{-1}])$.
and
$$
\widehat{\mathscr{F}}_{\mathbf{a}}^{\;(\sigma)}(t)\in W\langle t, t^{-1}, h_2(t)^{-1} \rangle, \quad h_2(t):=\prod\limits_{i=0}^{r-1}F_{\mathbf{a}^{(i)}}(t)_{<p}
$$
is convergent function, where $W\langle t, t^{-1}, h_2(t)^{-1} \rangle :=\underset{n}{\varprojlim}(W/p^n[t,t^{-1},h_2(t)^{-1}])$.
\end{cor}
\begin{proof}
    See \cite[Corollary 3.4]{A} and \cite[Corollary 3.2]{W}. 
\end{proof}

%%%%%%%%%%%%%%%%%%%%%%%%%%%%%%%%%%%%%%%%%%%%%%%%%%%%%%%%%%%%%%%%%%%%%%%%%%%%%%
\section{Congruence Relations for $B_k/A_k$ and $\widehat{B}_k/A_k$}
%%%%%%%%%%%%%%%%%%%%%%%%%%%%%%%%%%%%%%%%%%%%%%%%%%%%%%%%%%%%%
%%%%%%%%%%%%%%%%%%%%%%%%%%%%%%%%%%%%%%
%%%%%         fix a 
%%%%%%%%%%%%%%%%%%%%%%%%%%%%%%%%%%%%%%%
In what follows, we fix an $a\in\mathbb{Z}_p\backslash \mathbb{Z}_{\leq 0}$ and $\underline{a}=\mathbf{a}=(a,\cdots,a)\in \mathbb{Z}_p^s.$ For simplicity we write $F_a(t):=F_{\underline{a}}(t)$ (similar for $G^{(\sigma)}_{\underline{a}}(t),$ $\widehat{G}^{(\sigma)}_{\underline{a}}(t),$ $\mathscr{F}^{(\sigma)}_{\underline{a}}(t)$ and $\widehat{\mathscr{F}}^{\; (\sigma)}_{\mathbf{a}}(t)).$\medskip

Here we provide an explicit expression for the coefficient of these $p$-adic hypergeometric functions. If we write $F_{a}(t)=\sum A_{k}t^k$, $F_{a^\prime}(t)=\sum A_k^{(1)}t^k$, $G^{(\sigma)}_{a}(t)=\sum B^{(\sigma)}_kt^k$ and $\widehat{G}^{(\sigma)}_{a}(t)=\sum \widehat{B}_k^{(\sigma)}t^k$, then the coefficients are given by
$$
B_0^{(\sigma)}=s\psi_p(a)+s\gamma_p-p^{-1}\log(c), \quad B_k^{(\sigma)}=\frac{A_k-c^{k/p}A^{(1)}_{k/p}}{k}, \quad k\in \mathbb{Z}_{> 0}
$$
and
$$
\widehat{B}_k^{(\sigma)}=\frac{1}{k+a}\biggl(A_{k}-(-1)^{se}(A_{\frac{k-l}{p}}^{(1)})c^{\frac{k+a}{p}}\biggr),
$$
where $A_{\frac{m}{p}}^{(1)}=0$ if $m \not\equiv 0 \mod p$ or $m <0$.
In fact, we have $\mathscr{F}_{a}^{\;(\sigma)}(t)$, $G^{(\sigma)}_{a}(t)$, $\widehat{\mathscr{F}}_{a}^{\;(\sigma)}(t)$ and $\widehat{G}^{(\sigma)}_{a}(t)$ are power series with $W$-coefficients.(\cite[Lemma 2.2]{W})\medskip

Here we put $\sigma:W[[t]]\rightarrow W[[t]]$ and $ \widehat{\sigma}:W[[t]\rightarrow W[[t]]$ by $\sigma(t)=ct^p$ and $\widehat{\sigma}(t)=c^{-1}t^p$ where $c\in 1+pW.$ For simplicity, in what follows, we write $B_k$ for $B_k^{(\sigma)}$ and $\widehat{B}_k$ for $\widehat{B}_k^{(\widehat{\sigma})}.$ \medskip

Let us recall the following properties for $B_k/A_k$ and $\widehat{B}_k/A_k$.
\begin{thm}\label{congBA}
$B_k/A_k$ and $\widehat{B}_k/A_k$ are in the Witt ring $W(\overline{\mathbb{F}}_p)$. 
Furthermore, for any $k, k^\prime\in \mathbb{Z}_{\geq 0}$ and $n\in \mathbb{Z}_{>0}$, we have
$$
\frac{B_k}{A_k}\equiv \frac{B_{k^\prime}}{A_k} \mod p^n, \quad \frac{\widehat{B_k}}{A_k}\equiv \frac{\widehat{B}_{k^\prime}}{A_k} \mod p^n
$$ if $k\equiv k^\prime \mod p^n.$
\end{thm}
\begin{proof}
This is \cite[Lemma 3.9]{A} and \cite[Lemma 3.5]{W}. Although, in \cite[Lemma 3.9]{A}, the author only treated the case $c=1,$ it can be shown that (as in \cite[Lemma 3.7]{W}) that it is true for general $c\in 1+pW.$ 
\end{proof}

Let us define two functions $f:\mathbb{Z}_{>0}\rightarrow W$ and $\widehat{f}:\mathbb{Z}_{>0}\rightarrow W$ by 
$$
f(k)=\frac{B_k}{A_k}, \quad \widehat{f}(k)=\frac{\widehat{B}_k}{A_k}.
$$

In order to understand the transformation formulas between $p$-adic hypergeometric functions, we must further clarify the relationship between $f(k)$ and  
$\widehat{f}(k)$. By Theorem \ref{congBA} and using $p$-adic interpolation, we can extend functions $f$ and $\widehat{f}$ from $\mathbb{Z}_{>0}$ to $\mathbb{Z}_p$ denoted by $\beta$ and $\widehat{\beta},$ respectively. We write the value of $\beta$(resp. $\widehat{\beta}$) at $\lambda\in \mathbb{Z}_p$ by $\beta_\lambda$(resp. $\widehat{\beta}_{\lambda}$). \medskip

In fact, although it might not be apparent over the integers, we can find a close relationship between these two functions when viewing them over the $p$-adic numbers. The following lemma is one of the central results of this paper, which provides a crucial relation between $\beta$ and $\widehat{\beta}.$

%%%%%%%%%%%%%%%%%%%%%%%%%%%%%%%%%%%%%%%%%%%%%%%%%%%%%%%%%%%
%%%%    \beta +\widehat{\beta}=0
%%%%%%%%%%%%%%%%%%%%%%%%%%%%%%%%%%%%%%%%%%%%%%%%%%%%%%
\begin{lem}\label{beta+hbeta} Let $\lambda \in \mathbb{Z}_p$ and $a\in \mathbb{Z}_p\backslash \mathbb{Z}_{\leq 0},$ then we have
$$
\beta_{\lambda}+\widehat{\beta}_{-\lambda-a}=0.
$$
\end{lem}

Before giving the proof of Lemma \ref{beta+hbeta}, we need to introduce some notation and one more lemma.
\begin{defi}
For $\alpha\in\mathbb{Z}_p$ and $n\in\mathbb{Z}_{\geq 1},$ let us define 
$$
\{\alpha\}_n=\prod_{\substack{1\leq i\leq n \\ p\nmid (\alpha+i-1)}}(\alpha+i-1)
$$ and 
$\{\alpha\}_0=1.$\medskip
\end{defi}

\begin{lem}\label{for_beta_hbeta}
Let $a\in \mathbb{Z}_p\backslash \mathbb{Z}_{\leq 0}$ and let $x, y\in\mathbb{Z}_{\geq 0}$ satisfying $x+y+a\equiv 0 \mod p^{n}.$ Then
$$
(-1)^{f_x}\frac{\{1\}_x}{\{a\}_x}\equiv (-1)^{f_y}\frac{\{1\}_y}{\{a\}_y} \mod p^{n}.
$$ where $f_x=\ell_x-\lfloor \ell_x/p \rfloor,$ $f_y=\ell_y-\lfloor \ell_y/p \rfloor$ and $\ell_x, \ell_y\in \{0,1,\cdots, q-1\}$ such that $x-\ell_x\equiv y-\ell_y\equiv 0 \mod q.$
\end{lem}
\begin{proof}
We use $p$-adic Gamma functions and their functional equations to prove this lemma. \medskip

Observe that for any $\alpha\in \mathbb{Z}_p,$ one has 
$$
\{\alpha \}_n=(-1)^n\frac{\Gamma_p(\alpha+n)}{\Gamma_p(\alpha)}.
$$ After rewriting this lemma in term of $\Gamma_p$, we find it follows from the following fact.
\begin{prop}[\cite{Sch}, Proposition 37.2]
Let $x\in \mathbb{Z}_p.$ If $p\neq 2$ then
$$
\Gamma_p (x)\Gamma_p(1-x)=(-1)^{R(x)}
$$ where $R(x)\in \{1,\cdots,p\}$ and $R(x)\equiv x \mod p.$ \medskip

For $p=2,$ we have
$$
\Gamma_2(x)\Gamma_2(1-x)=(-1)^{\sigma_1(x)+1}
$$ where $\sigma_1$ is defined by the formula
$$
\sigma_1(\sum_{j=0}^\infty a_j2^j)=a_1.
$$
\end{prop}

\end{proof}

%%%%%%%%%%%%%%%%%%%%%%%%%%%%%%%%%%%%%%%%%%%%%%%
%%  proof of beta+hbeta
%%%%%%%%%%%%%%%%%%%%%%%%%%%%%%%%%%%%%%%%%%%%%%%
Now we can prove Lemma \ref{beta+hbeta}.
\begin{proof}[(proof of Lemma \ref{beta+hbeta})]
It suffices to prove 
$$
\beta_{\lambda}+\widehat{\beta}_{-\lambda-a}\equiv 0 \mod p^n
$$ for all $n\in\mathbb{Z}_{\geq 0}.$ Since $\beta_{\lambda}\equiv \beta_{\lambda^\prime}$ and $\widehat{\beta}_{\lambda}\equiv \widehat{\beta}_{\lambda^\prime} \mod p^n$ if $\lambda\equiv \lambda^\prime \mod p^n$, we can pick $b\in\mathbb{Z}_{>0}$ such that $b\equiv \lambda \mod p^n$ and consider $
\beta_{b}+\widehat{\beta}_{-b-a}$ instead. \medskip

If $p\nmid b$, then $\beta_b=1/b$ and $\widehat{\beta}_{-b-a}=-1/b.$ Hence, $\beta_b+\widehat{\beta}_{-b-a}=0.$ Now we suppose $p|b.$ Pick $x,y \in\mathbb{Z}_{>0}$ and $0< x, y\leq p^{2n}$ such that $x\equiv b \mod p^n$ but $p^{n+1} \nmid x$ and $x+y+a\equiv 0 \mod p^{2n}.$ Then we claim 
$$
\beta_x+\widehat{\beta}_y\equiv 0 \mod p^n.
$$

By \cite[Lemma 3.6]{A}, we have
$$
\frac{A^{(1)}_{x/p}}{A_x}=\bigg(\frac{\{1\}_x}{\{a\}_x}\bigg)^s, \quad \frac{A^{(1)}_{\frac{y-l}{p}}}{A_y}=\bigg(\frac{\{1\}_y}{\{a\}_y}\bigg)^s.
$$

Therefore $\beta_x+\widehat{\beta}_y$ becomes
\begin{equation}
\frac{1-c^{x/p}(\frac{\{1\}_x}{\{a\}_x})^s}{x}+\frac{1-(-1)^{se}c^{-\frac{y+a}{p}}(\frac{\{1\}_y}{\{a\}_y})^s}{y+a}. 
\end{equation}

Since $x+y+a\equiv 0 \mod p^{2n},$ we write $x+y+a=zp^{2n}$ for some $z\in \mathbb{Z}_p.$ Since $p^{n+1}\nmid x$, we have ${\rm ord}_p(y+a)={\rm ord}_p(-x+zp^{2n})\leq n.$ Therefore after modulo $p^n,$ we have
\begin{equation}\label{beta mod}
\begin{aligned}
&\frac{1-c^{x/p}(\frac{\{1\}_x}{\{a\}_x})^s}{x}+\frac{1-(-1)^{se}c^{-\frac{y+a}{p}}(\frac{\{1\}_y}{\{a\}_y})^s}{y+a} \\
\equiv &\frac{1-c^{x/p}(\frac{\{1\}_x}{\{a\}_x})^s}{x}+\frac{1-(-1)^{se}c^{-\frac{y+a}{p}}(\frac{\{1\}_y}{\{a\}_y})^s}{-x} \mod{p^n} \\
= &\frac{-c^{x/p}(\frac{\{1\}_x}{\{a\}_x})^s+c^{-\frac{y+a}{p}}((-1)^e\frac{\{1\}_y}{\{a\}_y})^s}{x}.
\end{aligned}
\end{equation}

Observe that $c^{-\frac{y+a}{p}}=c^{\frac{x}{p}-zp^{2n-1}}\equiv c^{x/p} \mod p^{2n}.$ Thus, in order to prove equation (\ref{beta mod}) is $0$ modulo $p^n$, it suffices to prove that
$$
\frac{\{1\}_x}{\{a\}_x}- (-1)^e\frac{\{1\}_y}{\{a\}_y}\equiv 0 \mod p^{2n}.
$$ This follows from Lemma \ref{for_beta_hbeta}. Therefore the result follows.
\end{proof}
%%%%%%%%%%%%%%%%%%%%%%%%%%%%%%%%%%%%%%%%%%%%%%%%%%%%
%%%%%        Transformation Formulas         %%%%%%%
%%%%%%%%%%%%%%%%%%%%%%%%%%%%%%%%%%%%%%%%%%%%%%%%%%%%
\section{Transformation Formulas}
In this section, we will recall two conjectures of $p$-adic hypergeometric functions. The first one is ``Transformation formulas between $\mathscr{F}_{a}^{(\sigma)}(t)$ and $\widehat{\mathscr{F}}_{a}^{(\sigma)}(t)$'' and the second one is ``Transformation formulas of Dwork's $p$-adic hypergeometric functions $\mathscr{F}^{{\rm Dw}}_{a}(t)$''. Our main result in this paper is to prove that the second conjecture implies the first one.

\subsection{Involution of $W\langle t,t^{-1},h(t)^{-1}\rangle$}\label{involution}

In order to properly formulate and establish the transformation formulas, it is necessary to first define a suitable involution of $W\langle t,t^{-1},h(t)^{-1}\rangle$. In this subsection, we recall such an involution.\medskip

Let $a_1, \cdots, a_{r-1}\in \mathbb{Z}_p$ and put $h(t):=\prod_{i=0}^{r-1}F_{a_i}(t)_{<p},$ where $F_{a_i}(t)$ is hypergeometric power series.\medskip

We define 
$$
W\langle t, t^{-1}, h(t)^{-1} \rangle :=\underset{n}{\varprojlim}(W/p^n[t,t^{-1},h(t)^{-1}]).
$$
Then there is an involution
$$
\omega : W\langle t,t^{-1},h(t)^{-1}\rangle \longrightarrow W\langle t,t^{-1},h(t)^{-1}\rangle, \quad \omega(f(t))=f(t^{-1}).
$$
This follows from the following proposition in \cite[Proposition 4.11]{W}.
\begin{prop}\label{involution}
$(1)$ Let $a\in \mathbb{Z}_p$ and $F(t):=F_{a}(t)_{<p} \mod p.$ Then \textcolor{black}{$F(t)=(-1)^{ls} t^lF(t^{-1})$} in $\mathbb{F}_p[t]$ where \textcolor{black}{$l$} is the degree of $F(t)$ which equals the unique integer in $\{0,1,\cdots,p-1\}$ such that \textcolor{black}{$a+l\equiv 0 \mod p.$} \medskip

$(2)$ Let $a_i\in \mathbb{Z}_p(0\leq i \leq r-1)$ and put $h(t):=\prod_{i=0}^{r-1}F_{a_i}(t)_{<p}.$ Then there is a ring homomorphism 
$$
\omega_n: W/p^n[t,t^{-1},h(t)^{-1}]\rightarrow W/p^n[t,t^{-1},h(t)^{-1}], \quad f(t)\mapsto f(t^{-1}).
$$\medskip
$(3)$ There is an involution
$$
\omega : W\langle t,t^{-1},h(t)^{-1}\rangle \longrightarrow W\langle t,t^{-1},h(t)^{-1}\rangle, \quad \omega(f(t))=f(t^{-1}).
$$ 
\end{prop}

%%%%%%%%%%%%%%%%%%%%%%%%%%%%%%%%%%%%%%%%%%%%
%%%%%%      Transformation Formula   %%%%%%%
%%%%%%%%%%%%%%%%%%%%%%%%%%%%%%%%%%%%%%%%%%%%
\subsection{Transformation Formula between $\mathscr{F}_{a}^{\;(\sigma)}(t)$ and $\widehat{\mathscr{F}}_{a}^{\;(\widehat{\sigma})}(t)$}

Using the involution $\omega$ defined in Section \ref{involution}, we can describe the transformation formulas that relates the object $\mathscr{F}_{a}^{\;(\sigma)}(t)$ to its counterpart $\widehat{\mathscr{F}}_{a}^{\;(\widehat{\sigma})}(t)$ under the inversion symmetry $t \mapsto t^{-1}$. Specifically, the following conjecture from \cite{W} predicts the algebraic relation between these two entities.
\begin{conj}[{\cite[Conjecture 4.12]{W}}, \textbf{Transformation Formula between $\mathscr{F}_{a}^{\;(\sigma)}(t)$ and $\widehat{\mathscr{F}}_{a}^{\;(\widehat{\sigma})}(t)$}] 
\label{transformation-conj}
Let $\sigma(t)=ct^p$ and $\widehat{\sigma}(t)=c^{-1}t^p$. Let $a\in\mathbb{Z}_p\backslash \mathbb{Z}_{\leq 0}$ and the $r$th Dwork prime $a^{(r)}=a$ for some $r>0.$ Put $h(t):=\prod_{i=0}^{r-1}F_{a^{(i)}}(t)_{<p},$ then
$$
\mathscr{F}_{a}^{\;(\sigma)}(t)=-\widehat{\mathscr{F}}_{a}^{\;(\widehat{\sigma})}(t^{-1})
$$
in the ring $W\langle t,t^{-1},h(t)^{-1}\rangle$ 
where $\widehat{\mathscr{F}}_{a}^{\;(\widehat{\sigma})}(t^{-1})$ is defined as $\omega(\widehat{\mathscr{F}}_{a}^{\;(\widehat{\sigma})}(t)).$ The map
$$
\omega: W\langle t,t^{-1},h(t)^{-1}\rangle \rightarrow W\langle t,t^{-1},h(t)^{-1}\rangle 
$$ is defined in Section \ref{involution}.
\end{conj}

\begin{remark}
  This conjecture is true modulo $p$ (cf. \cite[Theorem 4.13]{W}). 
\end{remark}

The conjecture under certain conditions was already established in \cite{W}.
\begin{thm}[\cite{W}, Theorem 4.14]
The conjecture is true for $s=2, a\in \frac{1}{N}\mathbb{Z}, 0<a<1$ and $p>N.$ 
\end{thm}

%%%%%%%%%%%%%%%%%%%%%%%%%%%%%%%%%%%%%%%%%%%%%%%%%%%%%%%%%%%%%%%%%%%%%%%%%%%
%% Transformation formula on Dwork's $p$-adic Hypergeometric Functions %%%%
%%%%%%%%%%%%%%%%%%%%%%%%%%%%%%%%%%%%%%%%%%%%%%%%%%%%%%%%%%%%%%%%%%%%%%%%%%%

\subsection{Transformation formula of Dwork's $p$-adic Hypergeometric Functions}\medskip
We now introduce the conjecture of transformation formulas of Dwork's $p$-adic hypergeometric functions. This conjecture was first proposed in \cite[Conjecture 4.18]{W}. However, as we will point out, the original formulation in \cite[Conjecture 4.18]{W} contains an  error(see Example \ref{Dwork-trans-thm}, and Theorem \ref{dw_trans_ncase}). To rectify this, we present a corrected version of the conjecture in this paper, and it is this revised formulation that we will investigate and build upon.

\begin{conj}[\textbf{Transformation formula on Dwork's $p$-adic Hypergeometric Functions}]\label{Dwork-trans-conj}
Let $a\in \Bbb{Z}_p$ and $l$ is the unique integer in $\{0,1,\cdots,p-1\}$ such that $a+l\equiv 0 \mod p.$ Then for odd prime $p,$ we have
$$
\mathscr{F}^{{\rm Dw}}_{a}(t)=((-1)^st)^{l}\mathscr{F}^{{\rm Dw}}_{a}(t^{-1}).
$$
For $p=2,$ we have
$$
\mathscr{F}^{{\rm Dw}}_{a}(t)=\pm((-1)^st)^{l}\mathscr{F}^{{\rm Dw}}_{a}(t^{-1}).
$$
\end{conj}

To illustrate the validity and concrete structure of this revised conjecture, we first present an explicit example for the case $s=1$.
\begin{eg2}\label{Dwork-trans-thm}
Let $p$ be an odd , $s=1$ and $a\in \mathbb{Z}_p,$ then we have
$$
\mathscr{F}^{{\rm Dw}}_{a}(t)=(-t)^{l}\mathscr{F}^{{\rm Dw}}_{a}(t^{-1}),
$$
where $l$ is the unique integer in $\{0,1,\cdots,p-1\}$ such that $a+l\equiv 0 \mod p.$\medskip

For $p=2$, we have
$$
\mathscr{F}^{{\rm Dw}}_{a}(t)=\pm(-t)^{l}\mathscr{F}^{{\rm Dw}}_{a}(t^{-1}).
$$
$(+ : a^\prime \equiv 0 \mod 2; -: a^\prime \equiv 1 \mod 2)$.
\end{eg2}
\begin{proof}
By \cite[p.37, Theorem 2]{Dw}, we have
$$
\mathscr{F}^{\rm Dw}_a(t)\equiv\frac{F_a(t)_{<p^n}}{F_{a^\prime}(t^p)_{<p^n}}= \frac{(1-t)^{-a}_{<p^n}}{(1-t^p)^{-a^\prime}_{<p^n}} \equiv \frac{(1-t)^{-N}_{<p^n}}{(1-t^p)^{-N^\prime}_{<p^n}} \mod p^n
$$ for some $N\in \mathbb{Z}_{>0}$ and $a\equiv N \mod p.$ \medskip

Then we obtain
$$
\begin{aligned}
\mathscr{F}^{\rm Dw}_a(t^{-1})&\equiv\left.\frac{(1-t)^{-N}_{<p^n}}{(1-t^p)^{-N^\prime}_{<p^n}}\right|_{t^{-1}} \mod p^n\\
&\equiv\left.\frac{(1-t)^{-N}}{(1-t^p)^{-N^\prime}}\right|_{t^{-1}} \\
&=\frac{(-t)^{-pN^\prime}}{(-t)^{-pN^\prime}}\cdot \frac{(1-t^{-1})^{-N}}{(1-t^{-p})^{-N^\prime}} \\
&=\frac{(-t)^{-l}(1-t)^{-N}}{(1-t^p)^{-N^\prime}}
\end{aligned}
$$ when $p$ is odd. Therefore the result follows. \medskip

For $p=2,$ using similar calculation, we have
$$
\begin{aligned}
\mathscr{F}^{\rm Dw}_a(t^{-1})&\equiv(-1)^{-N^\prime}\frac{(-t)^{-l}(1-t)^{-N}}{(1-t^2)^{-N^\prime}} \mod 2^n.
\end{aligned}
$$ Since $a^\prime \equiv N^\prime \mod 2,$ we obtain our result.
\end{proof}
\begin{remark}
    In the original statement proposed in \cite[Conjecture 4.18]{W}, it is conjectured that
    $$
\mathscr{F}^{{\rm Dw}}_{a}(t)=((-1)^st)^{l}\mathscr{F}^{{\rm Dw}}_{a}(t^{-1})
$$ without up to sign.
    However, it fails to hold (see Example \ref{Dwork-trans-thm}, and Theorem \ref{dw_trans_ncase}).
\end{remark}

The conjecture under certain conditions was previously established by the authors in \cite{W} and subsequently generalized by Nemoto \cite{N}. Specifically, we proved the following theorem.
\begin{thm}[\cite{W}, Theorem 4.20]\label{dw_trans_formula_wcase}
The conjecture is true for $s=2, a\in N^{-1}\mathbb{Z}, 0<a<1$ and $p>N.$ 
\end{thm}

Recently, Nemoto \cite{N} generalized Theorem \ref{dw_trans_formula_wcase} to arbitrary $s\geq 2$ and relaxed the condition.
\begin{thm}[\cite{N}, Theorem 1.2]\label{dw_trans_ncase}
Suppose that the following two conditions:
\begin{enumerate}
    \item  $a\in N^{-1}\Bbb{Z}$ and $p\nmid N$ for some $N\geq 2$
    \item $0<a<1$
\end{enumerate}

Then Conjecture \ref{Dwork-trans-conj} is true when $s\geq 2$. Furthermore, if $p=2$, $$
\mathscr{F}^{{\rm Dw}}_{a}(t)=-((-1)^st)^{l}\mathscr{F}^{{\rm Dw}}_{a}(t^{-1})
$$ holds if and only if $s$ is odd and $a^\prime \equiv 1 \pmod{2}$.
\end{thm}
%%%%%%%%%%%%%%%%%%%%%%%%%%%%%%%%%%%%%%%%%%%%%%
%% proof of conj_2  implies cong_1
%%%%%%%%%%%%%%%%%%%%%%%%%%%%%%%%%%%%%%%%%%%%%%
\subsection{The proof of Conjecture \ref{Dwork-trans-conj} implies Conjecture \ref{transformation-conj}}
In this section, we prove that Conjecture \ref{Dwork-trans-conj} implies Conjecture \ref{transformation-conj}. This conditional reduction serves as the bridge that allows us to deduce new cases for Conjecture \ref{transformation-conj} by utilizing recent developments.\medskip

First, we prove the following lemma. 
\begin{lem}\label{A-cong}
Let $a\in \mathbb{Z}_p,$ $m, k, d\in\mathbb{Z}_{\geq 0}$ with $0\leq m\leq p^n-1$ and $0\leq d \leq n.$ If Conjecture \ref{Dwork-trans-conj} is true, then we have
\begin{multline}\label{eq.2}
\begin{split}
    \sum_{\substack{i \equiv k \bmod{p^{n-d}} \\ 0\leq i \leq m \\ i+j=m}}A_iA
_{p^n-j-1}-\sum_{\substack{p^n-j^\prime-1 \equiv -k-a \bmod p^{n-d}\\ 0\leq j^\prime \leq m \\ i^\prime+j^\prime=m}}A_{i^\prime}A_{p^n-j^\prime-1}\equiv 0 \mod p^{d+1}.
\end{split}
\end{multline}
\end{lem}
\begin{proof}
The proof of this lemma is similar to \cite[Lemma 3.12]{A}. \medskip

Put 
$$
F^{(r)}(t):=\sum_{i=0}^{\infty}A_i^{(r)}t^i
$$ and
$$
F_k^{(r)}(t):=\sum_{i \equiv k \mod p^{n-d}}A_i^{(r)}t^i=p^{-n+d}\sum_{u=0}^{p^{n-d}-1}\zeta^{-ku}F(\zeta^u t)
$$ where $\zeta$ is a primitive $p^{n-d}$ root of unity. For simplicity, we write $F(t)$ and $F_k(t)$ for $F^{(0)}(t)$ and $F_k^{(0)}(t),$ respectively. \medskip

Since we suppose Conjecture \ref{Dwork-trans-conj} is true, we have
$$
\begin{aligned}
&\frac{F(t)}{F^{(n-d)}(t^{p^{n-d}})} \\
=&\frac{F(t)}{F^{(1)}(t^p)}\frac{F^{(1)}(t^p)}{F^{(2)}(t^{p^2})}\cdots \frac{F^{(n-d-1)}(t^{p^{n-d-1
}})}{F^{(n-d)}(t^{p^{n-d}})} \\
=&\pm((-1)^s t)^{l_0}\left.\frac{F(t)}{F^{(1)}(t^p)}\right|_{t^{-1}}\cdot (\pm)((-1)^s t^p)^{l_1}\left.\frac{F^{(1)}(t^p)}{F^{(2)}(t^{p^2})} \right|_{t^{-1}}\\ 
& \quad \cdots (\pm)((-1)^st^{p^{n-d-1}})^{l_{n-d-1}}\left.\frac{F^{(n-d-1)}(t^{p^{n-d-1
}})}{F^{(n-d)}(t^{p^{n-d}})}\right|_{t^{-1}} \\  
=&(\pm 1)^{n-d}(-1)^{s(l_0+l_1+\cdots+l_{n-d-1})} t^{l_0+l_1p+\cdots+l_{n-d-1}p^{n-l-1}}\left.\frac{F(t)}{F^{(1)}(t^p)}\right|_{t^{-1}}\cdots \left.\frac{F^{(n-d-1)}(t^{p^{n-d-1
}})}{F^{(n-d)}(t^{p^{n-d}})}\right|_{t^{-1}}
\end{aligned}
$$
where $l_0, l_1,\cdots, l_{n-d-1}$ are the numbers in $\{0,1,\cdots,p-1\}$ such that
$a+l_0\equiv a^\prime+l_1\equiv \cdots \equiv a^{(n-d-1)}+l_{n-d-1}\equiv 0 \mod p.$ \medskip

From the Dwork congruence \cite[p.37, Theorem 2]{Dw}, we have
$$
\frac{F^{(i)}(t)}{F^{(i+1)}(t^p)}\equiv \frac{F^{(i)}(t)_{<p^m}}{F^{(i+1)}(t^p)_{<p^m}} \mod p^n
$$ when $m\geq n \geq 1.$ Therefore one has
$$
\frac{F^{(i)}(t^p)}{F^{(i+1)}(t^{p^2})} \equiv \frac{F^{(i)}(t^p)_{<p^{n+1}}}{F^{(i+1)}(t^{p^2})_{<p^{n+1}}} \mod p^n,
$$
$$
\frac{F^{(i)}(t^{p^2})}{F^{(i+1)}(t^{p^3})} \equiv \frac{F^{(i)}(t^{p^2})_{<p^{n+2}}}{F^{(i+1)}(t^{p^3})_{<p^{n+2}}} \mod p^n, \cdots
$$ \medskip

After modulo $p^{d+1}$ and writing $s^\prime=s(l_0+l_1+\cdots+l_{n-d-1})$ and $l^\prime=l_0+l_1p+\cdots+l_{n-d-1}p^{n-d-1},$we have
$$
\begin{aligned}
&\frac{F(t)}{F^{(n-d)}(t^{p^{n-d}})}\\
\equiv &(\pm 1)^{n-d}(-1)^{s^\prime}t^{l^\prime}\left.\frac{F(t)_{<p^n}}{F^{(1)}(t^p)_{<p^n}}\right|_{t^{-1}}\cdots \left.\frac{F^{(n-d-1)}(t^{p^{n-d-1
}})_{<p^n}}{F^{(n-d)}(t^{p^{n-d}})_{<p^n}}\right|_{t^{-1}} \mod p^{d+1}\mathbb{Z}_p[[t]] \\
=&(\pm 1)^{n-d} (-1)^{s^\prime}t^{l^\prime}\left. \frac{F(t)_{<p^n}}{F^{(n-d)}(t^{p^{n-d}})_{<p^n}} \right|_{t^{-1}}.
\end{aligned}
$$
Therefore
\begin{equation}\label{eq.1}
\frac{F(t)}{F^{(n-d)}(t^{p^{n-d}})}=(\pm 1)^{n-d}(-1)^{s^\prime}t^{l^\prime}\left. \frac{F(t)_{<p^n}}{F^{(n-d)}(t^{p^{n-d}})_{<p^n}} \right|_{t^{-1}}+p^{d+1}\sum_{i=0}^\infty a_i t^i
\end{equation} where $a_i\in \mathbb{Z}_p.$ \medskip

If we substitute $\zeta^u t$ for $t,$ we have 
\begin{equation}\label{eq.2}
\begin{aligned}
\frac{F(\zeta^{u}t)}{F^{(n-d)}(t^{p^{n-d}})}&=(\pm 1)^{n-d}(-1)^{s^\prime}(\zeta^u t)^{l^\prime}\left. \frac{F(t)_{<p^n}}{F^{(n-d)}(t^{p^{n-d}})_{<p^n}} \right|_{\zeta^{-u} t^{-1}}+p^{d+1}\sum_{i=0}^\infty a_i \zeta^{ui}t^i \\
&= (\pm 1)^{n-d}(-1)^{s^\prime}(\zeta^u t)^{l^\prime}\left. \frac{F(\zeta^{-u} t)_{<p^n}}{F^{(n-d)}(t^{p^{n-d}})_{<p^n}} \right|_{t^{-1}}+p^{d+1}\sum_{i=0}^\infty a_i \zeta^{ui}t^i.
\end{aligned}
\end{equation} \medskip

Dividing equation (\ref{eq.2}) by equation (\ref{eq.1}) and expanding the quotient, one has
$$
F(\zeta^u t)\left .F(t)_{<p^n}\right|_{t^{-1}}\cdot t^{p^n-1}-F(t)F(\left.\zeta^{-u} t)_{<p^n}\right|_{t^{-1}}\cdot {t^{p^n-1}}\cdot \zeta^{ul^\prime}=p^{d+1}\sum_{i=0}^{\infty}b_i(\zeta^u)t^i
$$ where $b_i(x)\in\mathbb{Z}_p[x]$ are polynomials which does not depend on $u.$ \medskip

Taking $\sum\limits_{u=0}^{p^{n-d}-1}\zeta^{-uk}(-)$ on both side and noting that $l^\prime\in\{0,1,\cdots, p^{n-d}-1\}$ is the number such that $a+l^\prime\equiv 0\mod p^{n-d}$, one has
$$
p^{n-d}\bigg(F_k(t)\left .F(t)_{<p^n}\right|_{t^{-1}}\cdot t^{p^n-1}-F(t)F_{-k-a}(\left.t)_{<p^n}\right|_{t^{-1}}\cdot {t^{p^n-1}}\bigg)=p^{n-d}\cdot p^{d+1}g(x).
$$ for some power series $g(x).$ \medskip

Therefore
$$
\frac{F_k(t)}{F(t)}\equiv \left. \frac{F_{-k-a}(t)_{<p^n}}{F(t)_{<p^n}}\right|_{t^{-1}} \mod p^{d+1}.
$$
On the other hand, using a similar proof, one has
$$
\frac{F_k(t)}{F(t)}\equiv \frac{F_k(t)_{<p^n}}{F(t)_{<p^n}} \mod p^{d+1},
$$(cf. the last paragraph in the proof of \cite[Lemma 3.12]{A}) so we conclude
\begin{equation}\label{eq.3}
\frac{F_k(t)_{<p^n}}{F(t)_{<p^n}} \equiv \left. \frac{F_{-k-a}(t)_{<p^n}}{F(t)_{<p^n}}\right|_{t^{-1}} \mod p^{d+1}.
\end{equation}
Using (\ref{eq.3}), one has
\begin{equation}\label{eq.4}
  F_k(t)\left .F(t)_{<p^n}\right|_{t^{-1}}\cdot t^{p^n-1}=F(t)F_{-k-a}(\left.t)_{<p^n}\right|_{t^{-1}}\cdot {t^{p^n-1}} \mod p^{d+1}.  
\end{equation}

Then this lemma follows from comparing the coefficients of $t^m$ in (\ref{eq.4}).
\end{proof}

%%%%%%%%%%%%%%%%%%%%%%%%%%%%%%%%%%%%%%%%%%%%%%%%%%%%%
%%%%%%%%%%%%%%%    Theorem 3.6   %%%%%%%%%%%%%%%%%%%%
%%%%%%%%%%%%%%%%%%%%%%%%%%%%%%%%%%%%%%%%%%%%%%%%%%%%%
\begin{thm}\label{conj2_to_conj1}
Conjecture \ref{Dwork-trans-conj} implies Conjecture \ref{transformation-conj}.
\end{thm}
\begin{proof}
By congruence relations of $\mathscr{F}^{(\sigma)}_{a}(t)$ and $\widehat{\mathscr{F}}^{\;(\widehat{ \sigma})}_a(t)$ (Theorem \ref{cong_rel_1} and Theorem \ref{cong_rel_2}), Conjecture \ref{transformation-conj} is equivalent to the statement 
$$
\frac{G^{(\sigma)}_a(t)_{<p^n}}{F_a(t)_{<p^n}}\equiv -\left.\frac{\widehat{G}^{\; (\widehat{\sigma})}_a(t)_{<p^n}}{F_a(t)_{<p^n}}\right|_{t^{-1}} \mod p^nW[[t]]
$$ for all $n\in \mathbb{Z}_{\geq 0}.$ So it suffices to show that
$$
\sum_{\substack{i+j=m \\ 0\leq i, j \leq p^n-1}}B_iA_{p^n-j-1}+\widehat{B}_{p^n-j-1}A_i\equiv 0 \mod p^n.
$$ for any $m$ with $0\leq m\leq 2(p^n-1).$ \medskip

First, we suppose $0\leq m\leq p^n-1.$ Using the relation $\beta_k=B_k/A_k$ and $\widehat{\beta}_k=\widehat{B}_k/A_k$ when $k\in\mathbb{Z}_{\geq 0},$ we write
\begin{equation}\label{beta_equation}
\begin{aligned}
&\sum_{\substack{i+j=m \\ 0\leq i, j \leq m}}B_iA_{p^n-j-1}+\widehat{B}_{p^n-j-1}A_i \\
=&\sum_{\substack{i+j=m \\ 0\leq i,j \leq m}}\beta_iA_iA_{p^n-j-1}+\sum_{\substack{i^\prime+j^\prime=m \\ 0\leq i^\prime, j^\prime \leq m}}\widehat{\beta}_{p^n-j^\prime-1}A_{i^\prime}A_{p^n-j^\prime-1} \\
=& \sum_{k=0}^{p^n-1}\bigg[\sum_{\substack{i\equiv k \mod p^n \\ 0\leq i\leq m}}\beta_iA_iA_{p^n-j-1}+\sum_{\substack{p^n-j^\prime-1\equiv -k-a \mod p^n \\ 0\leq j^\prime \leq m}}\widehat{\beta}_{p^n-j^\prime-1}A_{i^\prime}A_{p^n-j^\prime-1}\bigg] \\
\equiv &\sum_{k=0}^{p^n-1}\beta_k\bigg[\sum_{\substack{i\equiv k\mod p^n \\ 0\leq i \leq m \\ i+j=m}}A_iA_{p^n-j-1}-\sum_{\substack{p^n-j^\prime-1\equiv -k-a \mod p^n\\ 0\leq j^\prime \leq m \\ i^\prime+j^\prime=m}}A_{i^\prime}A_{p^n-j^\prime-1}\bigg] \mod{p^n}.
\end{aligned}
\end{equation}

By Lemma \ref{A-cong}, we have
$$
\sum_{\substack{i\equiv k\mod p^n \\ 0\leq i \leq m \\ i+j=m}}A_iA_{p^n-j-1}-\sum_{\substack{p^n-j^\prime-1\equiv -k-a \mod p^n\\ 0\leq j^\prime \leq m \\ i^\prime+j^\prime=m}}A_{i^\prime}A_{p^n-j^\prime-1}\equiv 0 \mod p.
$$ So we can rewrite the equation (\ref{beta_equation}) as
$$
\sum_{k^\prime=0}^{p^{n-1}-1}\sum_{k\equiv k^\prime \mod p^{n-1}}\beta_k\bigg[\sum_{\substack{i\equiv k\mod p^n \\ 0\leq i \leq m \\ i+j=m}}A_iA_{p^n-j-1}-\sum_{\substack{p^n-j^\prime-1\equiv -k-a \mod p^n\\ 0\leq j^\prime \leq m \\ i^\prime+j^\prime=m}}A_{i^\prime}A_{p^n-j^\prime-1}\bigg]
$$
and this equals
$$
\begin{aligned}
&\sum_{k^\prime=0}^{p^{n-1}-1}\beta_{k^\prime}\sum_{k\equiv k^\prime \mod p^{n-1}}\bigg[\sum_{\substack{i\equiv k\mod p^n \\ 0\leq i \leq m \\ i+j=m}}A_iA_{p^n-j-1}-\sum_{\substack{p^n-j^\prime-1\equiv -k-a \mod p^n\\ 0\leq j^\prime \leq m \\ i^\prime+j^\prime=m}}A_{i^\prime}A_{p^n-j^\prime-1}\bigg] \\
&\equiv \sum_{k^\prime=0}^{p^{n-1}-1}\beta_{k^\prime} \bigg[\sum_{\substack{i\equiv k\mod p^{n-1} \\ 0\leq i \leq m \\ i+j=m}}A_iA_{p^n-j-1}-\sum_{\substack{p^n-j^\prime-1\equiv -k-a \mod p^{n-1} \\ 0\leq j^\prime \leq m \\ i^\prime+j^\prime=m}}A_{i^\prime}A_{p^n-j^\prime-1}\bigg] \mod p^n.
\end{aligned}
$$
Again by Lemma \ref{A-cong} and using the same argument, the equation equals
$$
\sum_{k^{\prime\prime}=0}^{p^{n-2}-1}\beta_{k^{\prime\prime}} \bigg[\sum_{\substack{i\equiv k\mod p^{n-2} \\ 0\leq i \leq m \\ i+j=m}}A_iA_{p^n-j-1}-\sum_{\substack{p^n-j^\prime-1\equiv -k-a \mod p^{n-2} \\ 0\leq j^\prime \leq m \\ i^\prime+j^\prime=m}}A_{i^\prime}A_{p^n-j^\prime-1}\bigg] \mod p^n.
$$

Continuing this process, we finally get the equation equals
$$
\sum_{k=0}^{p-1}\beta_k\bigg[\sum_{\substack{i\equiv k\mod p \\ 0\leq i \leq m \\ i+j=m}}A_iA_{p^n-j-1}-\sum_{\substack{p^n-j^\prime-1\equiv -k-a \mod p \\ 0\leq j^\prime \leq m \\ i^\prime+j^\prime=m}}A_{i^\prime}A_{p^n-j^\prime-1}\bigg]\equiv 0 \mod p^n.
$$

Thus
$$
\sum_{\substack{i+j=m \\ 0\leq i, j \leq m}}B_i A_{p^n-j-1}+\widehat{B}_{p^n-j-1}A_i\equiv 0 \mod p^n.
$$

On the other hand, if $p^n-1\leq m \leq 2(p^n-1)$, we put $i^\prime=p^n-i-1$ and $j^\prime=p^n-j-1$. Then
$$
\sum_{\substack{i+j=m \\ 0\leq i, j \leq p^n-1}}B_iA_{p^n-j-1}+\widehat{B}_{p^n-j-1}A_i
$$ becomes
$$
\sum_{\substack{i^\prime+j^\prime=2(p^n-1)-m \\ 0\leq i^\prime, j^\prime \leq p^n-1}}B_{p^n-i^\prime-1}A_{j^\prime}+\widehat{B}_{j^\prime}A_{p^n-i^\prime-1}. 
$$ Let $m^\prime=2(p^n-1)-m$ and interchange the role of $B_k$ and $\widehat{B}_k$, we find this is just the case above. Therefore the result follows. 
\end{proof}

\begin{cor}
Conjecture \ref{transformation-conj} is true when $s=1$.
\end{cor}
\begin{proof}
This follows from Theorem \ref{conj2_to_conj1} and Example \ref{Dwork-trans-thm}.
\end{proof}

\begin{cor}
Suppose that the following two conditions:
\begin{enumerate}
    \item  $a\in N^{-1}\Bbb{Z}$ and $p\nmid N$ for some $N\geq 2$
    \item $0<a<1$
\end{enumerate}

Then Conjecture \ref{transformation-conj} is true when $s\geq 2$.
\end{cor}
\begin{proof}
This follows from Theorem \ref{conj2_to_conj1} and Theorem \ref{dw_trans_ncase}.
\end{proof}

\end{document}